\documentclass[11pt,a4paper]{article}
\usepackage{mathrsfs, amsmath, amsthm, amssymb, bm, mathtools, thm-restate}
\usepackage{graphicx, xcolor, tikz}
\usetikzlibrary{calc, shapes}
\usepackage{caption, subcaption}
\usepackage{array}
\usepackage{cases}
\makeatletter
\def\th@plain{%
	\upshape 
}
\makeatother

\makeatletter
\renewenvironment{proof}[1][\proofname]{\par
	\pushQED{\qed}%
	\normalfont \topsep6\p@\@plus6\p@\relax
	\trivlist
	\item[\hskip\labelsep
	\bfseries
	#1\@addpunct{.}]\ignorespaces
}{%
\popQED\endtrivlist\@endpefalse
}
\makeatother

\newtheorem{theorem}{Theorem}
\newtheorem{lemma}{Lemma}
\newtheorem{corollary}{Corollary}
\theoremstyle{definition}
\newtheorem{definition}{Definition}
\newtheorem{proposition}{Proposition}

\newtheorem{observation}{Observation}

\usepackage[left=25mm, top=25mm, bottom=25mm, right=25mm]{geometry}
\usepackage[T1]{fontenc}
\usepackage[inline]{enumitem}
\usepackage[square, numbers, sort&compress]{natbib}

\usepackage[pdftex,%
bookmarks=true,%
bookmarksnumbered=true, 
bookmarksopen=true, 
plainpages=false,%
pdfpagelabels,%
colorlinks=true, 
linkcolor=blue, 
citecolor=blue,%
anchorcolor=green,
urlcolor=black,
breaklinks=true,
hyperindex=true]{hyperref}

\newcommand{\dd}{\textsf{Delete}}

\newcommand{\ds}{\textsf{DeleteSave}}

\newcommand{\del}{\textsf{Del}}

\def\int{\mathrm{int}}
\def\ext{\mathrm{ext}}

\def\fw{\mathrm{fw}}
\def\ew{\mathrm{ew}}
\def\w{\textsf{w}}

\usepackage[normalem]{ulem}
\begin{document}
	\title{Weak degeneracy of planar graphs and locally planar graphs}
	\author {Ming Han\footnote{College of Mathematics Science, Zhejiang Normal University, Jinhua, 321004, P. R. China}  {}\quad Tao Wang\footnote{Center for Applied Mathematics, Henan University, Kaifeng, 475004, P. R. China} {}\quad  Jianglin Wu\footnote{College of Mathematics Science, Zhejiang Normal University, Jinhua, 321004, P. R. China}  {}\quad Huan Zhou\footnote{College of Mathematics Science, Zhejiang Normal University, Jinhua, 321004, P. R. China} {}\quad Xuding Zhu\footnote{College of Mathematics Science, Zhejiang Normal University, Jinhua, 321004, P. R. China. This research is supported by Grants: NSFC 11971438, U20A2068, ZJNSF LD19A010001. }
	}
	
	\maketitle
	
	\begin{abstract}
		Weak degeneracy is a variation of degeneracy  which shares many nice properties of degeneracy. 
		In particular, if a graph $G$ is weakly $d$-degenerate, then for any $(d+1)$-list assignment $L$ of $G$, one can construct an $L$-coloring of $G$ by a modified greedy coloring algorithm. It is known that planar graphs of girth 5 are 3-choosable and locally planar graphs are $5$-choosable. This paper strengthens these results and proves that 
		planar graphs of girth 5 are weakly 2-degenerate and locally planar graphs are weakly 4-degenerate.

		
	\end{abstract}
	
	\section{Introduction}
	For a graph $G$, the greedy coloring algorithm colors vertices one by one in order $v_1, v_2, \dots, v_n$, assigning $v_i$ the least-indexed color not used on its colored neighbors. An upper bound for the number of colors used in such a coloring is captured in the notion of graph degeneracy. Let $\mathbb{Z}$ be the set of integers, and $\mathbb{Z}^G$ be the set of mappings $f: V(G) \to \mathbb{Z}$. For $f \in \mathbb{Z}^G$ and a subset $U$ of $V(G)$, let $f|_U$ be the restriction of $f$ to $U$, and let $f_{-U}: V(G)-U \to \mathbb{Z}$ be defined as $f_{-U}(x) = f(x) - |N_G(x) \cap U|$ for $x \in V(G)-U$. For convenience, we may use $f$ for $f|_U$, and  write $f_{-v}$ for $f_{-\{v\}}$. We denote by $E[U]$ the set of edges in $G$ with both end vertices in $U$. 
		
	Let $\mathcal{L}$ be the set of pairs $(G,f)$, where $G$ is a graph and $f \in \mathbb{Z}^G$. 
	\begin{definition}
		The \emph{deletion operation} $\dd(u): \mathcal{L} \to \mathcal{L}$ is defined as 
		$$\dd(u)(G,f)=(G-u,f_{-u}).$$
		We say $\textsf{Delete}(u)$ is \emph{legal} for $(G,f)$ if both $f$ and $f_{-u}$ are non-negative.  
		A graph $G$ is \emph{$f$-degenerate} if, starting with $(G,f)$, it is possible to remove all vertices from $G$ by a sequence of  legal deletion operations. For a positive integer $d$, we say that $G$ is \emph{$d$-degenerate} if it is degenerate with respect to the constant $d$ function. The \emph{degeneracy} of $G$, denoted by $\textsf{d}(G)$, is the minimum $d$ such that $G$ is $d$-degenerate.
	\end{definition}

	The quantity $\textsf{d}(G)+1$ is called the \emph{coloring number} of $G$, and is an upper bound for many graph coloring parameters: the chromatic number $\chi(G)$, the choice number $\chi_{\ell}(G)$, the paint number $\chi_\textsf{P}(G)$, the DP-chromatic number $\chi_\textsf{DP}(G)$ and the DP-paint number $\chi_\textsf{DPP}(G)$. The definitions of some of these parameters are complicated. As we  shall not discuss these parameters, other than saying that they are bounded by the weak degeneracy defined below, we omit the definitions and refer the reader to \cite{MR4117373} for the definitions and discussion about these parameters.

	The coloring number $\textsf{d}(G)+1$ of $G$, as an upper bound for the above mentioned graph coloring parameters, is often not tight. 
	It is therefore interesting to see if we can modify the greedy coloring algorithm to save  some of the colors and get a better upper bound. Motivated by this, Bernshteyn and Lee \cite{Bernshteyn2021a} recently introduced the concept of weak degeneracy of a graph.
	
	Assume $L$ is a list assignment of $G$ and we try to construct an $L$-coloring of $G$.  Assume $uw$ is an edge of $G$. In the greedy coloring algorithm, if we assign a color to $u$, then it is counted that $L(w)$ loses one color. However, if $|L(u)|>|L(w)|$, then one can assign to $u$ a color from $L(u)-L(w)$, and hence $L(w)$ will not lose a color in this step. The concept of weak degeneracy deals with this situation.

	\begin{definition}
			The \emph{deletion-save operation} $\ds(u,w): \mathcal{L} \to \mathcal{L}$ is defined as 
			$$\ds(u,w)(G,f)=(G-u,f_{-u}+\delta_w),$$
			where $\delta_w(v) = 1$ if $v=w$ and $\delta_w(v) = 0$ otherwise. 
			We say $\ds(u,w)$ is \emph{legal} for $(G,f)$ if $uw$ is an edge of $G$, $f(u) > f(w)$ and both $f$ and $f_{-u}+\delta_w$ are  non-negative.
	\end{definition}

	\begin{definition}
		A \emph{removal scheme} $\Omega = \del( \theta_1,\theta_2, \dots, \theta_k): \mathcal{L} \to \mathcal{L}$, where for each $i$, either $\theta_i=\langle u_i \rangle$ representing the deletion operation $\dd(u_i)$, or $\theta_i=\langle u_i,w_i \rangle$ representing the deletion-save operation $\ds(u_i,w_i)$, is defined recursively as follows:
$$\del(\langle u \rangle)(G,f)=\dd(u)(G,f), \quad\del(\langle u,w \rangle)(G,f)=\ds(u,w)(G,f)$$ and for $k \ge 2$, $$ \del(\theta_1,\theta_2, \dots, \theta_k)(G,f)= \del(\theta_k)(\del (\theta_1,\theta_2, \dots, \theta_{k-1})(G,f)).$$
We say $\del( \theta_1,\theta_2, \dots, \theta_k)$ is legal for $(G,f)$ if $\del (\theta_1,\theta_2, \dots, \theta_{k-1})$ is legal for $(G,f)$ and $\del(\theta_k)$ is legal for $ \del (\theta_1,\theta_2, \dots, \theta_{k-1})(G,f)$. Each $\theta_i$  {in a removal scheme} is called a {\em move}. A move $\theta_i=\langle u \rangle$ or $\theta_i=\langle u, w \rangle$ \emph{removes} $u$ from $G$. A graph $G$ is \emph{weakly $f$-degenerate} if there is a removal scheme $\Omega = \del( \theta_1,\theta_2, \dots, \theta_n)$ which is legal for $(G, f)$ and removes all vertices of $G$. For a positive integer $d$, we say that $G$ is \emph{weakly $d$-degenerate} if it is weakly degenerate with respect to the constant $d$ function. The \emph{weak degeneracy} of $G$, denoted by $\textsf{wd}(G)$, is the minimum $d$ such that $G$ is weakly $d$-degenerate.
    \end{definition}


	The following proposition was proved in \cite{Bernshteyn2021a}.
	
	\begin{proposition}
		For every graph $G$,
		$$\chi(G)\leq \chi_{\ell}(G)\leq \chi_{\textsf{DP}}(G)\leq \chi_{\textsf{DPP}}(G)\leq \textsf{wd}(G)+1.$$
	\end{proposition}
	
	Some well-known upper bounds for $\chi_\textsf{DP}(G)$ for families of graphs turn out to be upper bounds for $\textsf{wd}(G)+1$. For example, Bernshteyn and Lee \cite{Bernshteyn2021a} proved that planar graphs are weakly 4-degenerate and Brooks theorem remains true for weak degeneracy.

	It was proved by Thomassen \cite{MR1328294} that planar graphs of girth at least $5$ are 3-choosable. Dvo\v{r}\'{a}k and Postle \cite{MR3758240} observed that planar graphs with girth at least $5$ are DP-$3$-colorable. This paper strengthens this result and show that planar graphs of girth at least $5$ are weakly 2-degenerate. Indeed, we shall prove graphs in a slightly larger graph family are weakly 2-degenerate.
	
	We write $P= v_{1}v_{2}\dots v_{s}$ to indicate that $P$ is a path with vertices $v_1,v_2, \dots, v_s$ in this order, and write $K = (v_1v_2\dots v_k)$ to indicate that $K$ is a cycle with vertices $v_1,v_2, \dots, v_k$ in this cyclic order.  For convenience, we also denote by $P$ and $K$ the vertex sets of $P$ and $K$, respectively. The \emph{length} of a path or a cycle is the number of edges in the path or cycle.  A $k$-cycle (respectively, a $k^-$-cycle or a $k^+$-cycle) is a cycle of length $k$ (respectively, at most $k$ or at least $k$). Two cycles are \emph{adjacent} if they share some common edges, and we say they are \emph{normally adjacent} if their intersection is isomorphic to $K_{2}$. Let $\mathcal{G}$ denote the class of triangle-free plane graphs in which no 4-cycle is normally adjacent to a $5^-$-cycle. Dvo\v{r}\'{a}k, Lidick\'{y} and \v{S}krekovski \cite{MR2680225} proved that every graph in $\mathcal{G}$ is $3$-choosable. 
	In this paper, we prove that every graph in $\mathcal{G}$ is weakly $2$-degenerate. The proof uses induction, and for this purpose, we prove a stronger and more technical result.

	For a plane graph and a cycle $K$, we use $\int(K)$ to denote the set of vertices in the interior of $K$, and $\ext(K)$ to denote the set of vertices in the exterior of $K$. Denote by $\int[K]$ and $\ext[K]$ the subgraph of $G$ induced by $\int(K) \cup K$ and $\ext(K)  \cup K$, respectively. For the plane graph $G$, we denote by $B(G)$ the boundary walk of the infinite face of $G$.

	\begin{theorem}\label{MAINRESULT}
		Let $G \in \mathcal{G}$, and $P=p_{1}p_{2}\dots p_{s}$ be a path on $B(G)$ with at most   four vertices. 
		Let $f \in \mathbb{Z}^{G}$ be a function satisfying the following conditions:   
		\begin{enumerate}[label = (\roman*)]
			\item 
			$f(p_i)=0$   for $1 \le i \le s$,    
		  $f(v) = 2$ for all $v \notin B(G)$, and 
	 $1 \leq f(v) \leq 2$ for all $v \in B(G)\setminus V(P)$;
			\item ${I} = \{v \mid f(v) = 1\}$ is an independent set in $G$, and each vertex in ${I}$ has at most one neighbor in $P$. 
		\end{enumerate}
		Then $G-E[P]$ is weakly $f$-degenerate. 
	\end{theorem}

	The following is an easy consequence of Theorem~\ref{MAINRESULT}.

	\begin{corollary}
		Every graph in $\mathcal{G}$ is weakly $2$-degenerate. In particular, every planar graph of girth at least $5$ is weakly $2$-degenerate. 
	\end{corollary}
	
 {The proof of Theorem 
 \ref{MAINRESULT} uses induction, and follows a similar line  as the proof of the 3-choosability of these graphs in \cite{MR2680225}. Indeed, the idea of DeleteSave operation 
  was used in some cases in \cite{MR2680225} (as well as in many other papers on list colouring of graphs), although the term  DeleteSave   was not used explicitly.  Nevertheless,  the proof of Theorem  \ref{MAINRESULT}  requires rather different treatments in some cases. The conclusion that these graphs are weakly 2-degenerate is intrinsically  stronger. For example, it implies that these graphs are DP $3$-paintable, and  the proof in \cite{MR2680225} does not apply to DP-coloring. } 
	
	Assume $S$ is a surface and $G$ is a graph embedded in $S$. A cycle $C$ in $G$ is \emph{contractible} if, as a closed curve on $S$, it separates $S$ into two parts, and one part is homeomorphic to the disc. We say $C$ is \emph{non-contractible} otherwise. The length of the shortest non-contractible cycle in $G$ is called the \emph{edge-width} of $G$ and is denoted by $\ew(G)$. Note that if $S$ is the sphere, then every closed curve in $S$ is contractible, and hence $\ew(G) = \infty$ for any graph $G$ embedded in $S$. We say a graph $G$ embedded in a surface $S$ is \emph{``locally planar''} if $\ew(G)$ is ``large''. It was proved by Thomassen \cite{MR1234386} that for any surface $S$, there is a constant $\w$ such that any graph $G$ embedded in $S$ with $\ew(G) \ge \w$ is $5$-colorable. Roughly speaking, this result says that locally planar graphs are $5$-colorable. 
	This result was strengthened in a sequence of papers, where it was proved that locally planar graphs are $5$-choosable \cite{MR2462315}, $5$-paintable \cite{MR3351695} and DP $5$-paintable \cite{MR4117373}. In this paper, we further strengthen this result by proving the following result.
	
	\begin{theorem}
		\label{thm-main2}
		For any surface $S$, there is a constant $\w(S)$ such that every graph $G$ embedded in $S$ with edge-width at least $\w(S)$ is weakly $4$-degenerate.
	\end{theorem}

\section{Some preliminaries}
For $\Omega= \del( \theta_1,\theta_2, \dots, \theta_k)$ and $(G,f) \in \mathcal{L}$, let $$(G_{\Omega}, f_{\Omega}) = \del( \theta_1,\theta_2, \dots, \theta_k)(G,f).$$

If $\Omega=\del(\theta_1, \dots, \theta_k)$ and for each $i$, either $\theta_i= \langle u_i \rangle$ or $\theta_i=\langle u_i,w_i \rangle$, then let $U_{\Omega}=\{u_1,u_2,\dots, u_k\}$. Note that for any removal scheme $\Omega$, we have $G_{\Omega} = G-U_{\Omega}$ and $f_{\Omega} \ge f_{-U_\Omega}$. If $G[U]$ is weakly $f$-degenerate, then there is a removal scheme $\Omega$ legal for $(G[U],f)$ with $U_{\Omega}=U$.Thus we have the following observation. 

\begin{observation}
\label{obs-aa}
If $G-U$ is weakly $f_{-U}$-degenerate, then $G$ is weakly $f$-degenerate if and only if $G[U]$ is weakly $f$-degenerate. In particular, if $ f(x) \geq \deg_{G}(x)$, then $G$ is weakly $f$-degenerate if and only if $G-x$ is weakly $f$-degenerate.
\end{observation}

\begin{observation}
\label{obs-bb}
The following follows from the definition.
		 	\begin{enumerate}
		 	    \item
	            If $uv \notin E(G)$  and 
		 	  $\del(\langle u \rangle, \langle v \rangle )$ is legal for $(G,f)$, then $\del(\langle v \rangle, \langle u \rangle )$ is legal for $(G,f)$, and  $\del(\langle u \rangle, \langle v \rangle )(G,f) =  \del(\langle v \rangle, \langle u \rangle )(G,f)$.
		 	  	\item
		 	  	If $uv \notin E(G)$, and 
		 	  	$\del(\langle u,w \rangle, \langle v \rangle )$ is legal for $(G,f)$, then $\del(\langle v \rangle, \langle u,w \rangle )$ is legal for $(G,f)$, and  $\del(\langle u,w \rangle, \langle v \rangle )(G,f) =  \del(\langle v \rangle, \langle u,w \rangle )(G,f)$.
		 		\item
                If $\del(\langle u,v \rangle, \langle v \rangle )$ is legal for $(G,f)$, then $\del(\langle v \rangle, \langle u \rangle )$ is legal for $(G,f)$, and  $\del(\langle u, v \rangle, \langle v \rangle)(G,f) = \del(\langle v \rangle, \langle u \rangle )(G,f)$.    
		 	\end{enumerate}
\end{observation}

	\begin{proposition}
		\label{remark1}
		If $f(v) = 0$, then 
		$G$ is weakly $f$-degenerate if and only if $G-v$ is weakly $f_{-v}$-degenerate.
	\end{proposition}
	\begin{proof}
		If
		$\Omega=\del(\theta_1, \dots, \theta_n)$ is legal for $(G-v, f_{-v})$ and    removes all the vertices of $G-v$, then   
		$\Omega'=  \del(\langle v \rangle, \theta_1, \dots, \theta_n)$ is legal for $(G, f)$ and   removes all the vertices of $G$ since $\del(\langle v \rangle)(G,f)=(G-v, f_{-v})$.
		
		Conversely, assume that 
		$\Omega=\del(\theta_1, \dots, \theta_n)$ is legal for $(G, f)$ that  removes all the vertices of $G$. As 
		$f(v)=0$, $v$ is removed by a deletion operation and so there is an index $i$ such that $\theta_i= \langle v \rangle$.  
		For $j < i$, let  
		\[
		\theta'_j = \begin{cases} \langle u \rangle, &\text{ if $\theta_j=\langle u,v\rangle$}, \cr 
		\theta_j, &\text{ otherwise.}
		\end{cases}
		\]
		Note that if $uv \in E(G)$, and $u$ is removed in a move $\theta_j$ for some $j < i$, then since $f(v)=0$, we must have $\theta_j=\langle u,v \rangle$. By repeatedly applying Observation~\ref{obs-bb}, we conclude that $\Omega'=\del( \theta'_1, \dots, \theta'_{i-1},  \theta_{i+1},\dots, \theta_n)$ is legal for $(G-v, f_{-v})$ and removes all vertices of $G-v$. 
	\end{proof}

\section{Proof of Theorem~\ref{MAINRESULT}}
	It follows from Proposition~\ref{remark1} that  the conclusion of Theorem~\ref{MAINRESULT} 
	is equivalent to $G-P$ is weakly $f_{-P}$-degenerate. In the proof below, for different cases, we shall prove either of these two statements.

\begin{definition}
Assume $G$ is a plane graph and $P$ is a boundary path, and $f \in \mathbb{Z}^G$. We say $\Omega= \del( \theta_1,\theta_2, \dots, \theta_k)$ is \emph{legal} for $(G,P,f)$ if $\Omega$ is legal for $(G-P,f_{-P})$.
\end{definition}

	 It follows from the definition and Proposition~\ref{remark1} that if $\Omega$ is legal for $(G,P,f)$,  {and $G_{\Omega}-E[P]$ is weakly $f_{\Omega}$-degenerate}, then $G-E[P]$ is weakly $f$-degenerate.

	Assume Theorem~\ref{MAINRESULT} is not true, and $(G, P, f)$ is a counterexample with minimum $|V(G)| + |E(G)|$, and subject to this, with minimum $\sum_{v \in V(G)\setminus V(P)} f(v)$. To derive a contradiction, it suffices to find a removal scheme $\Omega$ legal for $(G,P,f)$, so that $(G_{\Omega}, P, f_{\Omega})$ satisfies the condition of Theorem~\ref{MAINRESULT}. Note that $\Omega$ is required to be legal for $(G,P,f)$, i.e., legal for $(G-P, f_{-P})$, and is not required to be legal for $(G,f)$. On the other hand, $\Omega$ is applied to $(G,f)$, and $G_{\Omega}$ contains the path $P$. Alternately, we may apply $\Omega$   to $(G-P,f_{-P})$. Then  to apply the induction hypothesis to the resulting graph, we need to change $(G-P)_{\Omega}$ back to $G_{\Omega}$ (i.e., add back the path $P$) and change $(f_{-P})_{\Omega}$ back to $f_{\Omega}$.

	\begin{lemma}\label{2CONNECTED}
		The graph $G$ is $2$-connected. 
	\end{lemma}
	\begin{proof}
		Suppose that $G$ has a cut-vertex $v$. Let $G_{1}$ and $G_{2}$ be two induced subgraphs of $G$ such that $V(G)= V(G_{1}) \cup V(G_{2})$ and $V(G_{1} \cap G_{2}) = \{v\}$ and $E(G)=E(G_1) \cup E(G_2)$. If $P \subseteq G_{1}$, then $(G_{1}, P, f)$ satisfies the conditions of Theorem~\ref{MAINRESULT}. Hence $G_{1}-P$ is weakly $f_{-P}$-degenerate. Also $(G_{2}, \{v\}, f)$ satisfies the conditions of Theorem~\ref{MAINRESULT}, and hence $G_{2} - v$ is weakly $f_{-v}$-degenerate. Note that $G-V(G_1)=G_2-v$ and the restriction of $f_{-V(G_1)}$ to $G_2-v$ equals $f_{-v}$. So by Observation~\ref{obs-aa}, $G - P$ is weakly $f_{-P}$-degenerate.

		Assume $P \nsubseteq G_{1}$ and $P \nsubseteq G_{2}$. Let $P_{1} = P \cap G_{1}$ and $P_{2} = P \cap G_{2}$. Then $v \in V(P)$, and $P_{i}$ is a path in $G_{i}$ for each $i \in \{1, 2\}$. Then $G_{1} - P_{1}$ and $G_{2} - P_{2}$ are weakly $f_{-P_{i}}$-degenerate. Hence $G - P$ is weakly $f_{-P}$-degenerate.
	\end{proof}

	It follows from Lemma~\ref{2CONNECTED} that the boundary $B(G)$ of $G$ is a cycle. A cycle $K$ in $G$ is \emph{separating} if both $\int(K)$ and $\ext(K)$ are not empty. 
	
	\begin{lemma}\label{SEPARATING}
		$|B(G)| \ge 8$ and every separating cycle in $G$ has length at least $8$. 
	\end{lemma}
	\begin{proof}		
		 Assume $B(G)=(v_1v_2\dots v_k)$ for some $k \le 7$. As $G$ is triangle-free and no $4$-cycle is normally adjacent to a $5^-$-cycle, $B(G)$ is an induced cycle.
		 
		 For convenience, assume $k=7$ (the $k \le 6$ can be treated similarly), and assume that $P=v_2v_3v_4v_5$. 
		  Let $G'=G-v_7$, and let $f'(v)=f(v) - 1$ for $v \in N_G(v_7)-\{v_1,v_6\}$, $f'(v_1)=f'(v_6)=1$ and $f'(v)=f(v)$ otherwise. It is easy to verify that $(G',P, f')$ satisfies the conditions of Theorem~\ref{MAINRESULT}, and hence $G' - P$ is weakly $f'_{-P}$-degenerate. As $f'_{-P}(v_1)=f'_{-P}(v_6)=0$, it follows from Proposition~\ref{remark1} that $G' - (P \cup \{v_1,v_6\})$ is weakly $f'_{-(P \cup \{v_1, v_6\})}$-degenerate.
		  Since $G-B(G)=G' - (P \cup \{v_1,v_6\})$ and $f_{-B(G)} = f'_{-(P \cup \{v_1, v_6\})}$, and $B(G)-P$ is certainly weakly $f_{-P}$-degenerate, it follows from Observation~\ref{obs-aa} that $G-P$ is weakly $f_{-P}$-degenerate. 
		  
		  Next we assume that $K$ is a separating $7^-$-cycle in $G$. Since $G \in \mathcal{G}$, $K$ is an induced cycle. 
		By the minimality of $G$, $\ext[K]-P$ is weakly $f_{-P}$-degenerate. 
		By Observation~\ref{obs-aa}, to show that $G-P$ is weakly $f_{-P}$-degenerate, it suffices to show that $G - \ext[K] = \int[K]-K$ is weakly $f_{-K}$-degenerate.

		Assume $K=(v_1,v_2, \dots, v_k)$, where $k \le 7$. 
		If $k \le 4$, then let $P'=(v_1,v_2, \dots, v_{k})$. Then $f'_{-K}=f'_{-P'}$. By the minimality of $G$, $\int[K]-K$ is weakly $f_{-K}$-degenerate. 
		Assume $k \ge 5$. Let $P'=v_1v_2\dots v_{k-3}$, $G'=\int[K]$
		and $f'' \in \mathbb{Z}^{G'}$ be defined as $f''(x)=0$ for $x \in P'$, $f''(v_{k-2})=f''(v_{k})=1$, $f''(v_{k-1})=2$ and $f''(x)=f(x)$ for $x \notin K$. By the minimality of $G$, $G'- P'$ is weakly $f''_{-P'}$-degenerate. The same argument as above shows that 
		$\int[K] - K$ is weakly $f_{-K}$-degenerate.
	\end{proof}

	\begin{lemma}\label{NOADJ}
		There are no $4$-cycles adjacent to $4$- or $5$-cycles. 
	\end{lemma}
	\begin{proof}
		Suppose to the contrary that a $4$-cycle $C_{1}$ is adjacent to a $5^{-}$-cycle $C_{2}$. 
		By assumption, $C_1$ and $C_2$ are not normally adjacent. So they intersect at three vertices. As $C_2$ has no chord, we may assume that $C_1=[a_1a_2a_3a_4]$ and $C_2=[a_1a_2a_3b_4]$ or $C_2=[a_1a_2a_3b_4b_5]$. 
		By Lemma~\ref{SEPARATING}, each of $C_{1}$ and $C_{2}$ bounds a face. Thus $a_{2}$ is a $2$-vertex which must be on the outer face. But then $[a_{1}a_{4}a_{3}b_4]$ or $[a_1a_4a_3b_4b_5]$ is a separating cycle of length at most $5$, contradicting Lemma~\ref{SEPARATING}. 
	\end{proof}

	A \emph{$k$-chord} of $B(G)$ is a path $Q$ of length $k$ such that only its two ends are on $B(G)$. A 1-chord is also called a chord of $B(G)$.
	Let $G_1,G_2$ be the two subgraphs with $V(G_1) \cap V(G_2) = V(Q)$ and $V(G_1) \cup V(G_2) = V(G)$ and $E(G)=E(G_1) \cup E(G_2)$. We say $G_1$ and $G_2$ are the subgraphs of $G$ \emph{separated} by $Q$. We index the subgraphs so that $|E(P \cap G_1)| \geq |E(P \cap G_2)|$. Hence $|E(P \cap G_2)| \le 1$. 
	
	\begin{observation}
		\label{ob2}
		 Let $P_2 = Q \cup (P \cap G_2)$. It is obvious that $(G_1,P,f)$ satisfies the conditions of Theorem~\ref{MAINRESULT}, and hence $G_1-P$ is weakly $f_{-P}$-degenerate. If $P_{2}$ is an induced path and $(G_2,P_2,f)$ also satisfies the conditions of Theorem~\ref{MAINRESULT}, then $G_2-P_2$ is weakly $f_{-P_2}$-degenerate, and it follows from Observation~\ref{obs-aa} that $G-P$ is weakly $f_{-P}$-degenerate, a contradiction. Thus we may assume that 
		$(G_2,P_2,f)$ does not satisfy the conditions of Theorem~\ref{MAINRESULT}.
	\end{observation}

	\begin{lemma}\label{NOCHORD}
		$B(G)$ has no chords. 
	\end{lemma}
	\begin{proof}
		Assume to the contrary that $B(G)$ has a chord $uw$. Let $G_1, G_2$ be the two subgraphs of $G$ separated by $uw$.
		
		Assume $P \subseteq G_{1}$. Since $G$ is triangle-free, each vertex in $G_2$ is adjacent to at most one vertex in $\{u, w\}$. Thus $(G_{2}, uw, f)$ satisfies the conditions of Theorem~\ref{MAINRESULT}, in contrary to Observation~\ref{ob2}. 
		
		Assume $P \nsubseteq G_{1}$ and $P \nsubseteq G_{2}$. Without loss of generality, assume that $w \in V(P)$. Then $|E[P] \cap E(B(G_i))| < |E[P]| \le 3$. Since $G$ is triangle-free, $u \notin V(P)$ and $P$ is an induced path.
		
		We may assume that $|E(P \cap G_2)| = 1$. If $P_{2}$ is not contained in a 4-cycle in $G_2$, then $(G_2, P_2, f)$ satisfies the conditions of Theorem~\ref{MAINRESULT}, a contradiction.

		Assume $P_{2}$ is contained in a 4-cycle in $G_2$. Since no 4-cycle in $G$ is adjacent to a $5^-$-cycle, $uw$ is not contained in a $5^-$-cycle in $G_1$. Let $P_1 = uw \cup (P \cap G_1)$. It is easy to verify that 
		$(G_{2}, P, f)$ and $(G_1, P_1, f)$ satisfy the conditions of Theorem~\ref{MAINRESULT}. By the minimality of $G$, $G_{2} - P$ is weakly $f_{-P}$-degenerate, and $G_{1} - P_{1}$ is weakly $f_{-P_{1}}$-degenerate. It follows from Observation~\ref{obs-aa} that $G - P$ is weakly $f_{-P}$-degenerate, a contradiction. 
	\end{proof}

	Since $B(G)$ is an induced cycle of length at least $8$ and $(G, P, f)$ is a counterexample with minimum $\sum_{v \in V(G)\setminus V(P)} f(v)$, we may assume that $P = p_{1}p_{2}p_{3}p_{4}$ is an induced path of length three.
	Assume $B(G)=p_{1}p_{2}p_{3}p_{4}x_{1}x_{2}\dots x_{m}$, where $m \ge 4$. We say a $k$-chord $Q$ of $B(G)$ \emph{splits off} a face ${F}$ from $G$ if one of the two subgraphs separated by $Q$ is the boundary cycle of $F$.

	\begin{lemma}\label{2-chord}
		Let $uvw$ be a $2$-chord of $B(G)$. Then $\{u, w\} \nsubseteq V(P)$, and $uvw$ splits off a $5^{-}$-face ${F}$ such that $|V({F}) \cap V(P)| \leq 2$. Moreover, if $|V({F}) \cap V(P)| \leq 1$, then ${F}$ is a $4$-face. Consequently, every internal vertex is adjacent to at most two vertices in $B(G)$ and adjacent to at most one vertex in $V(P)$. 
	\end{lemma}
	
	\begin{proof}
		Assume $uvw$ is a 2-chord and $G_{1}$ and $G_{2}$ are subgraphs separated by $uvw$. Assume $|E(P \cap G_{1})| > |E(P \cap G_{2})|$. 
		Let $P_2=uvw \cup (P \cap G_2)$. As $|E(P \cap G_{1})| > |E(P \cap G_{2})|$, we know that $|E(P \cap G_{2})| \le 1$ and hence $P_2$ has length $2$ or $3$. If $P_{2}$ is not induced, then since $G$ is triangle-free and $B(G)$ has no chord, $G[P_{2}]$ is a $4$-cycle that bounds a $4$-face by Lemma~\ref{SEPARATING}. So we may assume $P_{2}$ is an induced path. 
		
	By Observation~\ref{ob2}, $(G_2,P_2,f)$ does not satisfy the conditions of Theorem~\ref{MAINRESULT}. This means that $G_2$ has a vertex $y$ with $f(y)=1$ and $y$ is adjacent to two vertices of $P_2$. 
		
		If $P_2$ has length $2$, then $G_2$ is a 4-cycle, and hence $uvw$ splits off a 4-face. Moreover, $\{u,w\} \not\subseteq V(P)$, for otherwise, $G_1$ is a $5^-$-cycle, in contrary to Lemma~\ref{NOADJ}. 
		
		Assume $P_2$ has length $3$. We may assume $w=p_3$ and $P_2=uvp_3p_4$. As $B(G)$ has no chord, we know that either $y$ is adjacent to $p_4$ and $v$, or $y$ is adjacent to $p_4$ and $u$. 
		
		If $y$ is adjacent to $p_4$ and $v$, then $u \notin P$ (for otherwise $G_1$ is a 4-cycle and $G$ contains two adjacent 4-cycles). 
		Then $yvu$ is a 2-chord which separates $G$ into $G'_1$ and $G'_2$ with $P \subseteq V(G'_1)$. Let $P'_2=yvu$. Then $(G'_1,P, f)$ and $(G'_2, P'_2, f)$ satisfy the conditions of Theorem~\ref{MAINRESULT}, and hence $G'_1-P$ is weakly $f_{-P}$-degenerate, and $G'_2-P'_2$ is weakly $f_{-P'_2}$-degenerate (note that in this case, $G'_2$ is not a 4-cycle as $G$ contains no two adjacent 4-cycles). By Observation~\ref{obs-aa}, $G-P$ is weakly $f_{-P}$-degenerate. 
		
		Assume $y$ is adjacent to $p_4$ and $u$. Then $G_2$ is a facial 5-cycle  by Lemma~\ref{SEPARATING}. If $u \in P$, then $u=p_1$ and $G_1$ is a facial 4-cycle by Lemma~\ref{SEPARATING}, implying that $\deg_G(v)=2$, a contradiction. Thus $u \not\in P$ and $uvw$ splits off a $5$-face. 
	\end{proof}

	\begin{lemma}\label{2-chord-E}
		If $uvw$ is a 2-chord and $f(u) = 1$, then $w \in \{p_{2}, p_{3}\}$.
	\end{lemma}
	\begin{proof}
		Assume $uvw$ is a $2$-chord with $f(u) = 1$. By Lemma~\ref{2-chord}, the $2$-chord $uvw$ splits off a $5^{-}$-face $F = [uvw\dots xu]$ with $|V({F}) \cap V(P)| \leq 2$. Thus $\deg(x)=2$. Since $f(u)=1$, and $f(x) \leq \deg(x) - 1 = 1$, we conclude that $x \in P$ (as ${I}$ is an independent set). Since $F$ is a $5^{-}$-face, we have $x \in \{p_1,p_4\}$ and $w \in \{p_2,p_3\}$. 
	\end{proof}

	\begin{lemma}\label{3-chord}
		If $Q= uvwz$ is a $3$-chord with $\{u, z\} \cap \{p_{2}, p_{3}\} = \emptyset$, then $Q$ splits off a $5^{-}$-face ${F}'$ such that $V({F}') \cap V(P) \subseteq \{u, z\}$. 
	\end{lemma}
	\begin{proof}
		Let $G_{1}$ and $G_{2}$ be the two subgraphs of $G$ separated by $Q$. 
		Since $\{u, z\} \cap \{p_{2}, p_{3}\} = \emptyset$, we may assume that $P \subset G_{1}$.

		If $uz \in E(G)$, then since $B(G)$ has no chord, $G_2$ is a facial 4-cycle and $Q$ splits off a 4-face.
		Assume $uz \notin E(G)$.

		By Observation~\ref{ob2}, $(G_2,Q,f)$ does not satisfy the conditions of Theorem~\ref{MAINRESULT}. Then there exists a vertex $x$ with $f(x) = 1$ adjacent to two vertices in $\{u, v, w, z\}$. If $x$ is adjacent to $u$ and $z$, the $Q$ splits off the $5$-face ${F}'$ bounded by $uvwzxu$. 
		If $x$ is adjacent to $u$ and $w$, then by Lemma~\ref{2-chord}, the $2$-chord $xwz$ splits off a $4$-face $[xwzz'x]$. Then $uvwxu$ and $xwzz'x$ are adjacent 4-cycles, which contradicts Lemma~\ref{NOADJ}. The case that $x$ is adjacent to $z$ and $v$ is symmetric.
	\end{proof}

	We may assume that $f(x_1)=1$ or $f(x_2)=1$, for otherwise, let $f'=f$ except that $f'(x_1)=1$, $(G,P,f')$ satisfies the conditions of Theorem~\ref{MAINRESULT}. Note that ${I}$ is independent in $G$. Hence by the minimality of $\sum_{v \in V(G)\setminus V(P)} f(v)$, $G-P$ is weakly $f'_{-P}$-degenerate, and hence $G-P$ is weakly $f_{-P}$-degenerate.

	We write $f(x_1,x_2,\dots,x_i)=(a_1,a_2, \dots, a_i)$ to mean that $f(x_j)=a_j$ for $j=1,2,\dots, i$.

	Let $X$ be the set of boundary vertices defined as follows:
	
	\[
	X=\begin{cases}
	\{x_1\}, &\text{if $f(x_1, x_2,x_3)=(1,2,2)$}, \cr
	\{x_2,x_3,x_4\}, &\text{if $f(x_1,x_2,x_3,x_4) =(1,2,1,2)$ and either $m=4$ or $f(x_5)=1$}, \cr
	\{x_2,x_3,x_4,x_5\}, &\text{if $f(x_1,x_2,x_3,x_4, x_5) =(1,2,1,2,2)$ and either $m=5$ or $f(x_6)=1$}, \cr
	&\text{ and there is a 2-chord connecting $x_2$ and $x_4$}, \cr
		\{x_2,x_3\}, &\text{if $f(x_1,x_2,x_3,x_4, x_5) =(1,2,1,2,2)$ and either $m=5$ or $f(x_6)=1$}, \cr
		&\text{ and there is no 2-chord connecting $x_2$ and $x_4$}, \cr
	\{x_2,x_3,x_4\}, &\text{if $f(x_1,x_2,x_3,x_4, x_5, x_6) =(1,2,1,2,2,2)$}, \cr
	\{x_1, x_2,x_3\}, &\text{if $f(x_1,x_2,x_3,x_4) =(2,1,2,1)$}, \cr
	\{ x_2\}, &\text{if $f(x_1,x_2,x_3,x_4) =(2,1,2,2)$}. \cr
	\end{cases}
	\]
	Let
	$$Y=\{u: u \text{ is an interior vertex of a $3^-$-chord connecting two vertices of $X$}\}.$$

	Observe that if $x_jux_{j'}$ is a 2-chord with $j < j'$, then $j' \ge j+2$ as $G$ is triangle-free. On the other hand, for $j < j''< j'$, we have $\deg(x_{j''})=2$, as the 2-chord $x_jux_{j'}$ splits off a $5^-$-face and there is no separating $5^{-}$-cycle by Lemma~\ref{SEPARATING}. Hence $j'=j+2$ and $f(x_{j+1})=1$. Similarly, if $x_juvx_{j'}$ is a 3-chord with $j < j'$, then $j' \le j+2$ by Lemma~\ref{3-chord}, and if $j'=j+2$, then $f(x_{j+1})=1$ by Lemma~\ref{SEPARATING}.

	\begin{lemma}
		\label{lem-2}
		 	The following hold:
		 	\begin{enumerate}
		 		\item[(1)] No two vertices of $X \cup Y$ are connected by a $3^-$-path with interior vertices in $V(G)-(X \cup Y)$, where $3^-$-path means a path with length at most 3.
		 		\item[(2)] If $x_juvx_{j'}$ is a 3-chord connecting two vertices of $X$, then at most one of $u,v$ has a neighbor in $P$.				
		 		\item[(3)] Each vertex in $Y$ has degree $2$ in $G[X \cup Y]$. Hence $G[Y]$ consists of some isolated vertices and at most two copies of $K_2$.
		 		\item[(4)] There is no 2-chord $xuz$ with $x \in X$ and $z \notin X$ and $f(z)=1$.
		 	\end{enumerate}
	\end{lemma}
	\begin{proof}
		(1)-(3) can be easily checked in each case. We omit the details, but note that we may need to use the fact that $G$ is triangle-free, no $4$-cycle is adjacent to a $5^-$-cycle (Lemma~\ref{NOADJ}), and there is no separating $7^-$-cycle (Lemma~\ref{SEPARATING}). (4) follows from Lemma~\ref{2-chord-E}. 
	\end{proof}
	
	Assume that $Y=\{y_1, \dots, y_t\}$. Then it is easy to verify that $0 \le t \le 4$ ($Y =\emptyset$ if $t=0$). If $y_iy_j$ is an edge and $y_i$ is adjacent to a vertex in $P$, then we index the vertices of $Y$ so that $i < j$. 
	
	In the following, for convenience, we let $x_{m+1}=p_1$. 
	 
	 Let 
	 \[
	 \Omega =\begin{cases}
	\del (\langle x_1 \rangle), &\text{if $f(x_1, x_2,x_3)=(1,2,2)$}, \cr
	 \del (\langle x_4, x_5 \rangle, \langle x_3 \rangle, \langle x_2, x_1 \rangle, \langle y_1 \rangle, \dots, \langle y_t \rangle), &\text{if $f(x_1,x_2,x_3,x_4) =(1,2,1,2)$} \cr
	 &\text{ and either $m=4$ or $f(x_5)=1$}, \cr
	\del (\langle x_5, x_6 \rangle, \langle x_4 \rangle, \langle x_3 \rangle, \langle x_2, x_1 \rangle, \langle y_1 \rangle, \dots, \langle y_t \rangle), &\text{if $f(x_1,x_2,x_3,x_4, x_5) =(1,2,1,2,2)$} \cr
		&\text{ and either $m=5$ or $f(x_6)=1$, and } \cr
		&\text{ there is a 2-chord connecting $x_2$ and $x_4$}, \cr
\del (\langle x_3 \rangle, \langle x_2, x_1 \rangle, \langle y_1 \rangle, \dots, \langle y_t \rangle), 
		&\text{if $f(x_1,x_2,x_3,x_4, x_5) =(1,2,1,2,2)$} \cr
	    &\text{ and either $m=5$ or $f(x_6)=1$, and } \cr
		&\text{ there is no 2-chord connecting $x_2$ and $x_4$}, \cr
		\del (\langle x_4 \rangle, \langle x_3 \rangle, \langle x_2, x_1 \rangle, \langle y_1 \rangle, \dots, \langle y_t \rangle),  &\text{if $f(x_1,x_2,x_3,x_4, x_5, x_6) =(1,2,1,2,2,2)$}, \cr
	 \del (\langle x_3, x_4 \rangle, \langle x_2 \rangle, \langle x_1 \rangle, \langle y_1 \rangle, \dots, \langle y_t \rangle), &\text{if $f(x_1,x_2,x_3,x_4) =(2,1,2,1 )$}, \cr
		\del (\langle x_2 \rangle), &\text{if $f(x_1,x_2,x_3,x_4) =(2,1,2,2 )$}. \cr
	 \end{cases}
	 \]

	\begin{figure}[h!]%
		\def\s{1}
		\def\t{0.4*\s}
		\begin{tikzpicture}[scale = 0.8]
		\coordinate (O) at (0, 0);
		\coordinate (V1) at ($(O)+(\s, 0)$);
		\coordinate (V2) at ($(O)+(2*\s, 0)$);
		\coordinate (V3) at ($(O)+(3*\s, 0)$);
		\coordinate (V4) at ($(O)+(4*\s, 0)$);
		\coordinate (V5) at ($(O)+(5*\s, 0)$);
		\node[right] () at ($(O) + (0, -\s)$) {\del$(\langle x_{1} \rangle)$};
		\draw[dashed, red, fill = pink] ($(V1)+(-\t, \t)$)--($(V1)+(-\t, -\t)$)--($(V1) + (\t, -\t)$)--($(V1) + (\t, \t)$)--cycle;
		\draw (O)node[below]{\small$p_{4}$}--(V1)node[below]{\small$x_{1}$}--(V2)node[below]{\small$x_{2}$}--(V3)node[below]{\small$x_{3}$};
		\node[circle, inner sep = 1.5, fill, draw] () at (O) {};
		\node[rectangle, inner sep = 1.8, fill = white, draw] () at (V1) {};
		\node[regular polygon, regular polygon sides=3, inner sep = 0.9, fill = white, draw] () at (V2) {};
		\node[regular polygon, regular polygon sides=3, inner sep = 0.9, fill = white, draw] () at (V3) {};
		
		\begin{scope}[xshift = 11cm]
		\coordinate (O) at (0, 0);
		\coordinate (V1) at ($(O)+(\s, 0)$);
		\coordinate (V2) at ($(O)+(2*\s, 0)$);
		\coordinate (V3) at ($(O)+(3*\s, 0)$);
		\coordinate (V4) at ($(O)+(4*\s, 0)$);
		\coordinate (V5) at ($(O)+(5*\s, 0)$);
		\node[right] () at ($(O) + (0, -\s)$) {\del$(\langle x_{4}, x_{5} \rangle, \langle x_{3} \rangle, \langle x_{2}, x_{1} \rangle)$};
		\draw[dashed, red, fill = pink] ($(V2)+(-\t, \t)$)--($(V2)+(-\t, -\t)$)--($(V4) + (\t, -\t)$)--($(V4) + (\t, \t)$)--cycle;
		\draw (O)node[below]{\small$p_{4}$}--(V1)node[below]{\small$x_{1}$}--(V2)node[below]{\small$x_{2}$}--(V3)node[below]{\small$x_{3}$}--(V4)node[below]{\small$x_{4}$}--(V5)node[below]{\small$x_{5}$};
		\node[circle, inner sep = 1.5, fill, draw] () at (O) {};
		\node[rectangle, inner sep = 1.8, fill = white, draw] () at (V1) {};
		\node[regular polygon, regular polygon sides=3, inner sep = 0.9, fill = white, draw] () at (V2) {};
		\node[rectangle, inner sep = 1.8, fill = white, draw] () at (V3) {};
		\node[regular polygon, regular polygon sides=3, inner sep = 0.9, fill = white, draw] () at (V4) {};
		\node[circle, inner sep = 1.5, fill = white, draw] () at (V5) {};
		\end{scope}
		\end{tikzpicture}\vspace{0.5cm}
		
		\begin{tikzpicture}[scale = 0.8]
		\coordinate (O) at (0, 0);
		\coordinate (V1) at ($(O)+(\s, 0)$);
		\coordinate (V2) at ($(O)+(2*\s, 0)$);
		\coordinate (V3) at ($(O)+(3*\s, 0)$);
		\coordinate (V4) at ($(O)+(4*\s, 0)$);
		\coordinate (V5) at ($(O)+(5*\s, 0)$);
		\coordinate (V6) at ($(O)+(6*\s, 0)$);
		\coordinate (Z) at ($(V3)+(0, \s)$);
		\node[right] () at ($(O) + (0, -\s)$) {\del$(\langle x_{5}, x_{6} \rangle, \langle x_{4} \rangle, \langle x_{3} \rangle, \langle x_{2}, x_{1} \rangle)$};
		\draw[dashed, red, fill = pink] ($(V2)+(-\t, \t)$)--($(V2)+(-\t, -\t)$)--($(V5) + (\t, -\t)$)--($(V5) + (\t, \t)$)--cycle;
		\draw (O)node[below]{\small$p_{4}$}--(V1)node[below]{\small$x_{1}$}--(V2)node[below]{\small$x_{2}$}--(V3)node[below]{\small$x_{3}$}--(V4)node[below]{\small$x_{4}$}--(V5)node[below]{\small$x_{5}$}--(V6)node[below]{\small$x_{6}$};
		\draw (V2)--(Z)--(V4);
		\node[circle, inner sep = 1.5, fill, draw] () at (O) {};
		\node[rectangle, inner sep = 1.8, fill = white, draw] () at (V1) {};
		\node[regular polygon, regular polygon sides=3, inner sep = 0.9, fill = white, draw] () at (V2) {};
		\node[rectangle, inner sep = 1.8, fill = white, draw] () at (V3) {};
		\node[regular polygon, regular polygon sides=3, inner sep = 0.9, fill = white, draw] () at (V4) {};
		\node[regular polygon, regular polygon sides=3, inner sep = 0.9, fill = white, draw] () at (V5) {};
		\node[circle, inner sep = 1.5, fill = white, draw] () at (V6) {};
		\node[regular polygon, regular polygon sides=3, inner sep = 0.9, fill = white, draw] () at (Z) {};
		
		\begin{scope}[xshift = 11cm]
		\coordinate (O) at (0, 0);
		\coordinate (V1) at ($(O)+(\s, 0)$);
		\coordinate (V2) at ($(O)+(2*\s, 0)$);
		\coordinate (V3) at ($(O)+(3*\s, 0)$);
		\coordinate (V4) at ($(O)+(4*\s, 0)$);
		\coordinate (V5) at ($(O)+(5*\s, 0)$);
		\coordinate (V6) at ($(O)+(6*\s, 0)$);
		\coordinate (Z) at ($(V3)+(0, \s)$);
		\node[right] () at ($(O) + (0, -\s)$) {\del$(\langle x_{3} \rangle, \langle x_{2}, x_{1} \rangle)$};
		\draw[dashed, red, fill = pink] ($(V2)+(-\t, \t)$)--($(V2)+(-\t, -\t)$)--($(V3) + (\t, -\t)$)--($(V3) + (\t, \t)$)--cycle;
		\draw (O)node[below]{\small$p_{4}$}--(V1)node[below]{\small$x_{1}$}--(V2)node[below]{\small$x_{2}$}--(V3)node[below]{\small$x_{3}$}--(V4)node[below]{\small$x_{4}$}--(V5)node[below]{\small$x_{5}$}--(V6)node[below]{\small$x_{6}$};
		\draw[dashed] (V2)--(Z)--(V4);
		\node[circle, inner sep = 1.5, fill, draw] () at (O) {};
		\node[rectangle, inner sep = 1.8, fill = white, draw] () at (V1) {};
		\node[regular polygon, regular polygon sides=3, inner sep = 0.9, fill = white, draw] () at (V2) {};
		\node[rectangle, inner sep = 1.8, fill = white, draw] () at (V3) {};
		\node[regular polygon, regular polygon sides=3, inner sep = 0.9, fill = white, draw] () at (V4) {};
		\node[regular polygon, regular polygon sides=3, inner sep = 0.9, fill = white, draw] () at (V5) {};
        \node[circle, inner sep = 1.5, fill = white, draw] () at (V6) {};
        \node[regular polygon, regular polygon sides=3, inner sep = 0.9, fill = white, draw] () at (Z) {};
		\end{scope}
		\end{tikzpicture}\vspace{0.5cm}

		\begin{tikzpicture}[scale = 0.8]		
		\coordinate (O) at (0, 0);
		\coordinate (V1) at ($(O)+(\s, 0)$);
		\coordinate (V2) at ($(O)+(2*\s, 0)$);
		\coordinate (V3) at ($(O)+(3*\s, 0)$);
		\coordinate (V4) at ($(O)+(4*\s, 0)$);
		\coordinate (V5) at ($(O)+(5*\s, 0)$);
		\coordinate (V6) at ($(O)+(6*\s, 0)$);
		\node[right] () at ($(O) + (0, -\s)$) {\del$(\langle x_{4} \rangle, \langle x_{3} \rangle, \langle x_{2}, x_{1} \rangle)$};
		\draw[dashed, red, fill = pink] ($(V2)+(-\t, \t)$)--($(V2)+(-\t, -\t)$)--($(V4) + (\t, -\t)$)--($(V4) + (\t, \t)$)--cycle;
		\draw (O)node[below]{\small$p_{4}$}--(V1)node[below]{\small$x_{1}$}--(V2)node[below]{\small$x_{2}$}--(V3)node[below]{\small$x_{3}$}--(V4)node[below]{\small$x_{4}$}--(V5)node[below]{\small$x_{5}$}--(V6)node[below]{\small$x_{6}$};
		\node[circle, inner sep = 1.5, fill, draw] () at (O) {};
		\node[rectangle, inner sep = 1.8, fill = white, draw] () at (V1) {};
		\node[regular polygon, regular polygon sides=3, inner sep = 0.9, fill = white, draw] () at (V2) {};
		\node[rectangle, inner sep = 1.8, fill = white, draw] () at (V3) {};
		\node[regular polygon, regular polygon sides=3, inner sep = 0.9, fill = white, draw] () at (V4) {};
		\node[regular polygon, regular polygon sides=3, inner sep = 0.9, fill = white, draw] () at (V5) {};
		\node[regular polygon, regular polygon sides=3, inner sep = 0.9, fill = white, draw] () at (V6) {};
				
		\begin{scope}[xshift = 11cm]
		\coordinate (O) at (0, 0);
		\coordinate (V1) at ($(O)+(\s, 0)$);
		\coordinate (V2) at ($(O)+(2*\s, 0)$);
		\coordinate (V3) at ($(O)+(3*\s, 0)$);
		\coordinate (V4) at ($(O)+(4*\s, 0)$);
		\coordinate (V5) at ($(O)+(5*\s, 0)$);
		\coordinate (V6) at ($(O)+(6*\s, 0)$);
		\node[right] () at ($(O) + (0, -\s)$) {\del$(\langle x_{3}, x_{4} \rangle, \langle x_{2} \rangle, \langle x_{1} \rangle)$};
		\draw[dashed, red, fill = pink] ($(V1)+(-\t, \t)$)--($(V1)+(-\t, -\t)$)--($(V3) + (\t, -\t)$)--($(V3) + (\t, \t)$)--cycle;
		\draw (O)node[below]{\small$p_{4}$}--(V1)node[below]{\small$x_{1}$}--(V2)node[below]{\small$x_{2}$}--(V3)node[below]{\small$x_{3}$}--(V4)node[below]{\small$x_{4}$};
		\node[circle, inner sep = 1.5, fill, draw] () at (O) {};
		\node[regular polygon, regular polygon sides=3, inner sep = 0.9, fill = white, draw] () at (V1) {};
		\node[rectangle, inner sep = 1.8, fill = white, draw] () at (V2) {};
		\node[regular polygon, regular polygon sides=3, inner sep = 0.9, fill = white, draw] () at (V3) {};
		\node[rectangle, inner sep = 1.8, fill = white, draw] () at (V4) {};
		\end{scope}
		\end{tikzpicture}\vspace{0.5cm}
		
		\begin{tikzpicture}[scale = 0.8]
		\coordinate (O) at (0, 0);
		\coordinate (V1) at ($(O)+(\s, 0)$);
		\coordinate (V2) at ($(O)+(2*\s, 0)$);
		\coordinate (V3) at ($(O)+(3*\s, 0)$);
		\coordinate (V4) at ($(O)+(4*\s, 0)$);
		\coordinate (V5) at ($(O)+(5*\s, 0)$);
		\coordinate (V6) at ($(O)+(6*\s, 0)$);
		\node[right] () at ($(O) + (0, -\s)$) {\del$(\langle x_{2} \rangle)$};
		\draw[dashed, red, fill = pink] ($(V2)+(-\t, \t)$)--($(V2)+(-\t, -\t)$)--($(V2) + (\t, -\t)$)--($(V2) + (\t, \t)$)--cycle;
		\draw (O)node[below]{\small$p_{4}$}--(V1)node[below]{\small$x_{1}$}--(V2)node[below]{\small$x_{2}$}--(V3)node[below]{\small$x_{3}$}--(V4)node[below]{\small$x_{4}$};
		\node[circle, inner sep = 1.5, fill, draw] () at (O) {};
		\node[regular polygon, regular polygon sides=3, inner sep = 0.9, fill = white, draw] () at (V1) {};
		\node[rectangle, inner sep = 1.8, fill = white, draw] () at (V2) {};
		\node[regular polygon, regular polygon sides=3, inner sep = 0.9, fill = white, draw] () at (V3) {};
		\node[regular polygon, regular polygon sides=3, inner sep = 0.9, fill = white, draw] () at (V4) {};
		\end{tikzpicture}
		\caption{The definition of the subset $X$ and the corresponding removal scheme. A square indicates a vertex $v$ with $f(v) = 1$, a triangle indicates a vertex $v$ with $f(v) = 2$.}
	\end{figure}

It is straightforward to verify that $\Omega$ is legal for $(G, P, f)$ by Lemma~\ref{lem-2}~(2).   To finish the proof of Theorem~\ref{MAINRESULT}, it suffices to prove that $(G_{\Omega}, P, f_{\Omega})$ satisfies the conditions of Theorem~\ref{MAINRESULT}, and hence 
$G_{\Omega}-E[P]$ is weakly $f_{\Omega}$-degenerate.  

Assume $v \in B(G_{\Omega})$. If $f(v) =2$, then by Lemma~\ref{lem-2}~(1), $f_{\Omega}(v) \ge 1$. If $f(v)=1$, then since $G$ has no chord (Lemma~\ref{NOCHORD}) and there is no 2-chord $xuv$ with $x \in X$ and $v \notin X$ and $f(v)=1$ (Lemma~\ref{lem-2}~(4)), then $f_{\Omega}(v) = f(v) = 1$. So $(G_{\Omega}, P, f_{\Omega})$ satisfies the  (i) of Theorem~\ref{MAINRESULT}.

Let ${I}' = \{v \in B(G_{\Omega}): f_{\Omega}(v) = 1\}$.
Assume there exists $uv \in E(G_{\Omega})$ with $u,v \in {I}'$.
As ${I}$ is independent in $G$, we may assume that $f(u) = 2$. By Lemma~\ref{lem-2}~(1), $f(v) = 1$.
By Lemma~\ref{lem-2}~(4), $u$ is adjacent to a vertex $y\in Y$. Assume $yx\in E(G)$ and $x\in X$, then $vuyx$ is a 3-chord in $G$ with $\{v, x\} \cap \{p_{2}, p_{3}\} = \emptyset$.
 By Lemma~\ref{3-chord}, the $3$-chord $vuyx$ splits off a $5^{-}$-face $F$ with $V(F) \cap V(P) = \emptyset$. Since ${I}$ is independent in $G$ and $f(v) = 1$, the face $F$ must be a $4$-face $[vuyxv]$. Otherwise there is a separating 5-cycle. However, $yx$ must be in a $5^-$-cycle in $G[X\cup Y]$, a contradiction. 
Thus ${I}'$ is an independent set. 
By Lemma~\ref{2-chord}, every internal vertex of $G$ is adjacent to at most one vertex in $P$. So $(G_{\Omega}, P, f_{\Omega})$ satisfies the  (ii) of Theorem~\ref{MAINRESULT}.
	 So $(G_{\Omega}, P, f_{\Omega})$ satisfies the all conditions of Theorem~\ref{MAINRESULT}. This completes the proof of Theorem~\ref{MAINRESULT}.

\section{Proof of Theorem~\ref{thm-main2}}

The \emph{face-width} $\fw(G)$ of a graph $G$ embedded in a surface $S$ is the largest integer $k$ such that every non-contractible closed curve in $S$ intersects $G$ in at least $k$ points. It is obvious that for any graph $G$ embedded in $S$, $\fw(G) \le \ew(G)$. Theorem~\ref{thm-main2} will follow from the following Theorem~\ref{thmfw}.

\begin{theorem}\label{thmfw}
	For every surface $S$ there exists a constant $\w(S)$ such that
	every $5$-connected graph that can be embedded in $S$ with face-width at least $\w(S)$ is weakly $4$-degenerate.
\end{theorem}

Now we show that Theorem~\ref{thmfw} implies Theorem~\ref{thm-main2}. The following result was proved in \cite{Bernshteyn2021a} and is used in our proof.

\begin{lemma}\label{planarlem}
	Let $G$ be a plane graph with at least $3$ vertices, $P$ is a set of 2 consecutive vertices on $B(G)$. Let $f: V(G) \rightarrow \mathbb{Z}$ by
	$$f(u)=\begin{cases}
	    0, &\text{if $u \in P$;}\\
		2, &\text{if $u \in B(G) - P$;}\\
		4, &\text{otherwise.}
	\end{cases}$$
	Then $G - E[P]$ is weakly $f$-degenerate, or equivalently, $G-P$ is weakly $f_{-P}$-degenerate.
\end{lemma}

\begin{proof}[Proof of Theorem~\ref{thm-main2}]
Assume Theorem \ref{thmfw} is true.  Let $\w(S)$ be the constant in Theorem~\ref{thmfw}. We shall prove that every graph embedded in $S$ with edge-width at least $3\w(S)$ is weakly $4$-degenerate.

Assume this is not true and $G_0$ is a counterexample with minimum $|V(G_0)|$.
We construct a triangulation $G$ of $S$ as follows: For each $6^-$-face of $G_0$, add edges to triangulate it.
For each $7^+$-face, we add a \emph{chimney} as follows: Assume the boundary of the face is $C_0=(v_{0,1}, v_{0,2}, \dots, v_{0,k})$. Note that $C_{0}$ is not necessarily a cycle.
Add $k$ new cycles $C_i=(v_{i,1}, v_{i,2}, \dots, v_{i,k})$ ($i=1,2,\dots, k$), add edges $v_{i,j}v_{i+1,j}, v_{i,j}v_{i+1, j+1}$ ($i=0,1, \dots, k-1$, $j=1,2,\dots, k$, where the additions are carried out modulo $k$), and finally add a new vertex adjacent to all vertices of $C_k$.

It is easy to see (cf. \cite{MR2462315}) that $\ew(G)= \fw(G) \ge \frac13 \ew(G_0) \ge \w(S)$.
As $G_0$ is a subgraph of $G$ and $G_0$ is not weakly $4$-degenerate, $G$ is not weakly $4$-degenerate. 
Thus, by Theorem~\ref{thmfw}, $G$ is not $5$-connected.
As $G$ is a triangulation of $S$ and $\ew(G)$ is large, there is a contractible separating $4^-$-cycle $C$.
Let $D=\int[C]$ and $D^o = D-C$ be the interior of $D$.
We choose $C$ so that $D^o$ contains the minimum number of vertices (subject to
the condition that $D^o \ne \emptyset$). This implies that either $D^o$ contains a single vertex, or each vertex in $D^o$ is adjacent to at most two vertices
of $C$. Observe that in constructing $G$ from $G_0$, new vertices are added only when chimneys are added.
As chimneys are added to $7^+$-faces, the interior of $D$
cannot contain added vertices only. Therefore, $D':=G_0 \cap D^o \ne \emptyset$.

Let $G'_0 = G_0-D'$ and $f(v)=4$ for $v\in G_0$. As $\ew(G_{0}') \geq \ew(G_0) \geq 3\w(S)$, by the minimality of $G_0$, the graph $G'_0$ is weakly $f$-degenerate.

Next, we show that $G_{0} - G_{0}'=D'$ is weakly $f_{-G'_0}$-degenerate. If $D'$ has a single vertex $v$, then $f_{-G'_0}(v)\ge 0$ since $N_{G'_0}(v)\subseteq V(C)$ and $|C|\le 4$. Assume that $|D'|\ge 2$. Let $B(D')$ be the boundary of $D'$. Then each vertex $v\in B(D')$ 
is adjacent to at most two neighbors in $G'_0$, and so $f_{-G'_0}(v)\ge 2$. Every interior vertex $v\in V(D'-B(D'))$ is not adjacent to any vertex of $G'_0$, and hence $f_{-G'_0}(v)=4$.
By Lemma~\ref{planarlem}, $D'$ is weakly $f_{-G'_0}$-degenerate. 

By Observation~\ref{obs-aa}, $G_0$ is weakly $f$-degenerate, a contradiction.
\end{proof}

It remains to prove Theorem~\ref{thmfw}. We may assume that $G$ is a triangulation of the surface $S$, because adding edges does not decrease the face-width or the connectivity of a graph. First, we will recall the definition of nice $H$-scheme in \cite{MR3351695}, which is an important structure in our proof.

\begin{definition}
	\label{scheme}
	Assume $G$ is a graph embedded in $S$ and $H$ is a cubic graph.
	An \emph{$H$-scheme} in $G$ is a family ${\cal F}$ of induced subgraphs of $G$ together with a
	labeling which associates subgraphs in ${\cal F}$ to vertices and edges of $H$
	such that the following hold:
	\begin{itemize}
		\item[A1]   ${\cal F} =\{D(x): x \in V(H)\} \cup \{D(e): e \in E(H)\} \cup \{P(e,x): e \in E(H),x \in e\}$ consists of a family of   subgraphs of $G$, each   embedded in a disk in $S$, and for   $e \in E(H)$ and   $x \in e$,   $P(e,x)  $ is a path connecting a vertex $v_{e,x}$ on the boundary of $D(x)$
		to a vertex $u_{e,x}$ on the boundary of $D(e)$.
		\item[A2] The subgraphs in ${\cal F}$ are pairwise disjoint, except that $v_{e,x}$ belong to both $P(e,x)$ and $D(x)$, and
		$u_{e,x}$ belong to both $P(e,x)$ and $D(e)$.
		Also no edge of $G$ connects vertices of distinct subgraphs in ${\cal F}$, except that
		$v_{e,x}$ has neighbors in both $D(x)$ and $P(e,x)$, and $u_{e,x}$ has neighbors in both $D(e)$ and $P(e,x)$.
		\item[A3] By contracting each $D(x)$ into a single vertex for each $x \in V(H)$, and replace each $P(e,x)\cup D(e) \cup P(e,y)$ by an edge joining $x$ and $y$ for each edge $xy \in E(H)$, we obtain a $2$-cell embedding of $H$ in $S$.
	\end{itemize}
\end{definition}

For $e=xy\in E(H)$, let $v_{e,x}^-$ and $v_{e,x}^+$ be the two neighbors of $v_{e,x}$ on the boundary of the outer face of $D(x)$, $u_{e,x}^-$ and $u_{e,x}^+$ be the two neighbors of $u_{e,x}$ on the boundary of the outer face of $D(e)$, and $u'_{e,x}$ be the unique common neighbor of $u_{e,x}, u_{e, x}^+, u_{e,x}^-$ in $D(e)$, if such a vertex exists. Let
\begin{eqnarray*}
	D'(e) = D(e) - \{u_{e,x}, u^-_{e,x}, u^+_{e,x}, u'_{e,x}, u_{e,y}, u^-_{e,y}, u^+_{e,y}, u'_{e,y} \}.
\end{eqnarray*}
If the vertices $u'_{e,x}$ and/or $u'_{e,y}$ do not exist, then ignore them in the above formula. Let $P'(e,x)=P(e,x) \cup \{v_{e,x}^-, u_{e,x}^-\}$. A \emph{segment} of a path $P$ is subset of its vertices that induces a subpath of $P$.

Assume $\mathcal{F}$ is an $H$-scheme in a graph $G$ embedded in $S$.  {Let $U = \bigcup_{F \in {\cal F}}V(F)$, $U' = V(G) \setminus U$ and $G'=G[U']$. By A3, each component of $G'$ is a plane graph embedded in a disk on $S$.} Let $R$  {be} the bipartite subgraph of $G$ induced by edges between $U$ and $U'$.

The following lemma is a combination of Lemma 2.4 and Lemma 3.2 in \cite{MR3351695}.

\begin{lemma}
	\label{keylem}
	For any surface $S$, there is a constant $\w(S)$ such that the following holds:
	If a $5$-connected triangulation $G$ of $S$ has face-width at least $\w(S)$, then there is a cubic graph $H$ such that $G$ has an $H$-scheme $\mathcal{F}$  {satisfying} for each edge $e=xy$ of $H$, for any vertex $u \in D'(e)$, $$N_G(u) \cap U \subseteq D(e), |N_G(u) \cap \{u_{e,x}, u_{e,x}^+, u_{e,x}^-, u'_{e,x}, u_{e,y}, u_{e,y}^+, u_{e,y}^-, u'_{e,y}\}| \le 2,    {\text{dist}}_{G}(u_{e,x}, u_{e,y}) \geq 5.$$ Moreover, the associated bipartite graph $R$ has an orientation for which the following hold:
	\begin{itemize}
		\item[(1)] For $v \in U$, $\deg_R^+(v) \le 1$. Moreover, for $e \in E(H)$ and $x \in e$, if $v \in D(x)$ or $v \in D(e)- \{u_{e,x}, u_{e,x}^-, u_{e,y}, u_{e,y}^-\}$, then $\deg_R^+(v)=0$;
		\item[(2)] For $v \in U'$, $\deg_R^+(v) \le 2$.
	    \item[(3)] For $e \in E(H)$,  $x \in e$ and $v \in U'$, if $N_R(v) \cap P'(e,x) \ne \emptyset$, then $N_R(v) \cap V(G)$ is a segment of $P'(e,x)$, and if $av$ is an in-edge of $v$, then $v$ has two out-edges $vb,vc$ such that $b,c$ lies between $v_{e,x}$ and $a$ on $P'(e,x)$. 
     \end{itemize}
\end{lemma}

\begin{proof}[Proof of Theorem~\ref{thmfw}]
Assume $S$ is a surface, $\w(S)$ is the constant in Lemma~\ref{keylem}, $G$ is a 5-connected triangulation of $S$ with face-width at least $\w(S)$. Let ${\cal F}$ be a $H$-scheme in $G$, and $R$ be the associated bipartite graph oriented as in Lemma~\ref{keylem}. Let $f(u)=4$ for any vertex $u\in V(G)$, and let
	\begin{align*}
		G_1&=\bigcup_{x\in V(H)}D(x),\\
		G_2&=\bigcup_{e\in E(H),\ x\in e}(P(e,x) \cup \{u^{-}_{e,x}\})\setminus\{v_{e,x}\},\\
		G_3&=\bigcup_{e\in E(H),\ x\in e}\{u'_{e,x},u^+_{e,x}\},\\
		G_4&=\bigcup_{e\in E(H)}D'(e),\\
		G_5&=G[U'].
	\end{align*}
	Observe that $V(G_1), V(G_2), V(G_3), V(G_4), V(G_5)$ form a partition of $V(G)$. 
	As each component of $G_1$ is a plane graph $D(x)$ for some $x\in V(H)$, $G_1$ is a plane graph and so $G_1$ is weakly $f$-degenerate. 
	
	Note that each component of $G_2$ is a path $(P(e,x) \cup \{u^{-}_{e,x}\})\setminus \{v_{e,x}\}$ for some $e\in E(H),\ x\in e$. We order the vertices of $G_2$ as $x_1,x_2, \ldots, x_t$ so that if $i<j$ and $x_i, x_j \in (P(e,x)\cup \{u^-_{e,x}\})\setminus\{v_{e,x}\}$, then $x_i$ lies between $v_{e,x}$ and $x_j$ on $P'(e,x)$. By (A2) of Definition \ref{scheme},   $f_{-G_1}(x_i)=3$ if $x_i$ is the unique vertex of $P(e,x)$ adjacent to $v_{e,x}$ and $f_{-G_1}(x_i)=4$ otherwise. By Lemma \ref{keylem}, $\deg_R^+(x_i) \le 1$ for $i\in[t]$. Let $w_i\in N_R^+(x_i)$ if $\deg_R^+(x_i)=1$ for $i\in [t]$. We define a removal scheme $\Omega=\del(\theta_1,\theta_2,\dots,\theta_t)$ as follows: for $i\in [t]$,
	 \[
	\theta_i =\begin{cases}
		\langle x_i \rangle, &\text{if $N_R^+(x_i)=\emptyset$}, \cr
		\langle x_i, w_i \rangle, &\text{if $N_R^+(x_i)=\{w_i\}$} \cr
	\end{cases}
	\]
	It suffices to show that $\theta_i$ is legal for $i\in [t]$. Let $A_i=\{x_1,x_2,\dots,x_i\}$ for $i\in [t]$. If $\deg_R^+(x_i)=0$, then $f_{-(G_1\cup A_{i-1})}(x_i)\ge 3\ge 0$, and so it is legal. Otherwise, $\deg_R^+(x_i)=1$. By~(3) of Lemma~\ref{keylem}, $w_i$ has two removed out-neighbors. Thus, $f_{-(G_1\cup A_{i-1})}(w_{i})\le 2<3\le f_{-(G_1\cup A_{i-1})}(x_i)$, and so it is also legal.
	
	Now we consider the vertices in $G_3$. By definition and Lemma \ref{keylem},  each component of $G_3$ has at most two vertices, $u^+_{e,x}$ and $u'_{e,x}$ (if exists). As for each $e\in E(H), x\in e$, $u^+_{e,x}$  has only one neighbor in $G_1\cup G_2$ and $u'_{e,x}$ (if exists) has two removed neighbors in $G_1\cup G_2$, $f_{-(G_1\cup G_2)}(v)\ge 2$ for $v\in G_3$. Then for each component, we can remove its vertices by deletion operation in the order $u^+_{e,x}$, $u'_{e,x}$. With the same operation, we can remove all the vertices of $G_3$ legally since the components of $G_3$ do not affect each other.
	
	In the subgraph $G_4$, as each component of $G_4$ is a plane graph $D'(e)$ for some $e\in E(H)$, $G_4$ is a plane graph. By Lemma \ref{keylem}, for each edge $e=xy$ of $H$, for any vertex $u \in D'(e)$, $$N_G(u) \cap \left(\bigcup_{i=1}^3 V(G_i)\right) \subseteq \{u_{e,x}, u_{e,x}^+, u_{e,x}^-, u'_{e,x}, u_{e,y}, u_{e,y}^+, u_{e,y}^-, u'_{e,y}\},$$ and $|N_G(u) \cap \{u_{e,x}, u_{e,x}^+, u_{e,x}^-, u'_{e,x}, u_{e,y}, u_{e,y}^+, u_{e,y}^-, u'_{e,y}\}| \le 2$.
	Therefore, $f_{-(\bigcup_{i=1}^3G_i)}(u)\ge 2$ if $u$ is on the boundary of $G_4$, and $f_{-(\bigcup_{i=1}^3G_i)}(u)=4$ otherwise. By Lemma~\ref{planarlem}, $G_4$ is weakly $f_{-(\bigcup_{i=1}^3G_i)}$-degenerate. 
	
	Let $h$ be the function produced by legal removing all the vertices in $\bigcup_{i=1}^4G_i$. Finally, for each vertex $u$ on the boundary of $G_5$, $h(u) \ge 2$ since the value of $f(u)$ only decreases when its out-neighbor is removed and $u$ has at most two out-neighbors in $\bigcup_{i=1}^4G_i$. For the interior vertex $u$, $h(u)=f(u)=4$. Similarly, by Lemma~\ref{planarlem}, $G_5$ is weakly $h$-degenerate.
\end{proof}

As a consequence of Theorem~\ref{thm-main2}, the condition in Theorem~\ref{thmfw} on the connectivity is redundant.

\begin{corollary}
	For any surface $S$ there is a constant $\w(S)$ such that every graph $G$ embedded in $S$ with face-width at least $\w(S)$
	is weakly $4$-degenerate.
\end{corollary}

\end{document}